\documentclass[11pt]{amsart}
\usepackage{txfonts,color}
\usepackage{amssymb}
\theoremstyle{plain}

\renewcommand{\epsilon}{\varepsilon}

\newcommand{\bef}{\begin{flushright}}
\newcommand{\eef}{\end{flushright}}

\newcommand{\eval}[2][\right]{\relax
\ifx#1\right\relax \left.\fi#2#1\rvert}

\def\beq{\begin{eqnarray*}}
\def\eeq{\end{eqnarray*}}

\newtheorem{thm}{Theorem}

\newcommand{\bj}{\bar{j}}
\newcommand{\bk}{\bar{k}}
\newcommand{\bl}{\bar{l}}
\newcommand{\bm}{\bar{m}}

\newcommand{\bq}{\bar{q}}

\newcommand{\bz}{\bar{z}}

\newcommand{\ol}{\overline}

\begin{document}

\title[global Chern-Calabi flow]
{On global Chern-Calabi flows on compact complex surfaces}

\author{ Li Sheng, Li Ma}

  \address{Li Sheng, Department of Mathematics\\
Sichuan University \\
Chengdu  610065, Sichuan \\
China} \email{lsheng@scu.edu.cn}

\address{Li Ma, School of Mathematics and Physics \\
University of Science and Technology Beijing\\
Xueyuan Road 30, Haidian\\
Beijing 100083 \\
China}
\email{lma17@ustb.edu.cn}

\begin{abstract}
In this paper, we propose the Chern-Calabi Flow for extremal metrics in conformal Hermitian geometry. After the existence of extremal metric, we consider the existence of global flows to the Chern-Calabi flows for extremal metrics on compact complex surfaces with Gauduchon metrics.
To achieve this goal, we use the Adams inequality and the Gronwall inequality method to obtain fourth order derivative norm estimates.

 We also study the convergence of the global solution to the Chern-Calabi flow and its limit is the solution of the Chern-Calabi extremal metric problem. We propose a fourth order elliptic problem and derive the Bochner formula relevant to the elliptic estimates for the Chern-Calabi heat flow.


\end{abstract}

\maketitle

\section{Introduction}\label{sect1}

In the seminal paper \cite{Calabi-1}, E.Calabi introduced the extremal metric and the Calabi flow in a Kaehler class on a Kaehler manifold. The constant scalar curvature Kaehler metrics in Kaehler manifolds are extremal metrics and the properties of the extremal metrics are interesting topics in differential geometry. Since there are many interesting results in this direction, we just mention the paper \cite{Li-18}
 and refer to the book \cite{Sz-14} for more results and refrences. X.X.Chen (see section 3 in \cite{Chen-1}) found that the extremal metric on a Riemann surface is the critical metric in the conformal class with unit volume. Later, M.Struwe studied the heat flow approach to the Calabi functional on Riemann surfaces. It is natural question to study the the extremal of the Calabi functional on a conformal class of the Kaehler metric on a Kaehler manifold. In \cite{Chang}, the author studied the Calabi flow in the conformal class of the Riemannian metric on the four dimensional Einstein manifold.
In this paper, we consider the extremal metrics and study the Chern-Calabi flow on compact complex surfaces in the conformal Hermitian geometry. We shall set up the global existence of the flow. We also discuss the convergence of the global flow on compact complex surfaces.

Given a Hermitian manifold $(M^n,\omega,J)$, $n=dim_{\mathbb C} M$. Let $R^{ch}(\omega)$ be the Chern-Scalar curvature and we introduce the Chern-Calabi functional
$$
\mathbf C(\omega)=\frac{1}{n!}\int_M R^{ch}(\omega)^2 \omega^n.
$$
on the unit volume conformal class $\{\omega\}_1$. For $n=2$, the extremal metric $exp(f)\omega$ satisfies the equation
$$
(-\Delta^{ch}_{\omega})^*\Big(-n\Delta^{ch}_{\omega} f
+ R^{ch}(\omega)\Big)=0.
$$
In this case, the Chern-Calabi flow on compact complex surfaces is simply defined by
$$
 f_t=exp(-2f)\Big((\Delta^{ch}_{\omega})^{*}(-n\Delta^{ch}_{\omega} f+ R^{ch}(\omega)) \Big)
$$
 with the initial data $\omega_0\in \{\omega\}_1$.

Roughly speaking, our main results are below. We first prove the existence of the extremal metric on the conformal class $\{\omega\}_1$ of unit volume. Namely,
on a compact complex Hermitian surface $(M,\omega)$
with the metric $\omega$ being the Gauduchon metric.
Then there is a unique solution (up to a constant) of Chern-Calabi extremal metric problem  in the conformal class $\{\omega\}_1$.
For the study of Chern-Calabi flow on the compact complex surface, we prove that
For a given compact complex Hermitian surface $(M,\omega)$
with $\omega$ is a Gauduchon metric.
Then there is a global flow to Chern-Calabi flow on the compact complex surface. The proof of global flow depends on some hard analysis, which is motivated by the use of Adams (higher dimensional Moser-Trudinger) inequality in the paper of S.Brendle \cite{B}.

\subsection{Useful facts and Chern-Calabi flow}

We fist fix some notations and introduce the Calabi functional and extremal metrics for Chern connections.
Let $M^n$ be a complex manifold of complex dimension $n$
and let $g$ be a Riemannian metric on $M$. Let $J$ be the induced
almost complex structure on $M$ so
$J^2 = - \mbox{id}$ and $J$ is integrable. We assume that $J$ is compatible with the metric $g$, i.e.
$$
 g (u, v) = g (Ju, Jv), \;\; u, v \in TM.
$$

Let $\omega$ be the fundamental (i.e., {\em K\"ahler}) form of $g$ defined
by
$$
\omega (u, v) =  g (Ju, v).
$$

In local coordinates $(z_1, \ldots, z_n)$ on $M$, we set
$$
 g_{j\bk} = g \Big(\frac{\partial}{\partial z_j}, \frac{\partial}{\partial \bz_k}\Big),
  \;\; (g^{i\bj}) = (g_{i\bj})^{-1},
$$
and
$$
 \omega = \sqrt{-1} g_{j\bk} dz_j \wedge d \bz_k.
$$

Recall that the Chern-scalar curvature is
$$
 R^{ch}(\omega)= - g^{i\bj}\frac{\partial^2}{\partial z_i \partial \bz_j}\log\det g_{k\bl}.
 $$

For any real smooth function $f:M\to R$, we define the Chern-Laplacian operator of $f$ by
 $$
 Lf:=\Delta^{ch}_{\omega}f=g^{i\bj}\frac{\partial^2}{\partial z_i \partial \bz_j} f.
 $$
 Then
 $$
 R^{ch}(\omega)= -\Delta^{ch}_{\omega} \log\det g_{k\bl}.
 $$
We make the following notations on $(M,\omega)$ that $d\mu=\frac{\omega^n}{n!}$, $H^k(M)$ is the standard Sobolev space on $(M,\omega)$,
$$
V=\int_M d\mu, \ \ \overline{f}=\frac{1}{V}\int_M f d\mu,
$$
and $\langle \phi, \psi\rangle_f=\int_M \phi \psi e^fd\mu$ for any real $\phi, \psi\in L^2(M)$. $C$ is the positive constant which may change from line to line.

We have the well-known elliptic estimates that
$$
\| f- \overline{f}\|_{H^2} \leq C\|Lf\|_{L^2}
$$
and
\begin{equation}\label{ell}
\| f\|_{H^4} \leq C(\|L^*Lf\|_{L^2}+\|f\|_2),
\end{equation}
where $f\in C^\infty(M,R)$.

Recall that for $\omega(t)=exp(f(t))\omega=exp(f)\omega$,
\begin{equation}\label{chern-sc}
R^{ch}(\omega(t))=exp(-f)(-nL f+ R^{ch}(\omega))
\end{equation}
and for any $\phi\in C^\infty$,
$$
\Delta^{ch}_{\omega(t)}\phi=exp(-f)L\phi.
$$

The Calabi functional on compact complex manifolds is
$$
\mathbf C(\omega)=\frac{1}{n!}\int_M R^{ch}(\omega)^2 \omega^n.
$$
Here we take $\omega$ to be a Gauduchon metric (i.e., $\partial\overline{\partial }\omega^{n-1}=0$ ) and the existence of a Gauduchon metric is guaranteed in \cite{Ga}.
To compute the Euler-Lagrange equation of the functional $\mathbf C(\omega)$ in conformal geometry, we consider the conformal class $\{\omega\}=\{exp(f)\omega; \ f\in C^\infty(M)\}$ of $\omega$ and its volume one subclass
 $$
 \{\omega\}_1=\{\widetilde{\omega}:=exp(f)\omega; \ f\in C^\infty(M), \ \ \int_M \frac{\widetilde{\omega}^n}{n!}=1\}.
 $$
 For $\omega(t)=exp(f(t))\omega\in \{\omega\}$, the Calabi functional for $\omega(t)$ is
$$
\mathbf C(\omega(t))=\frac{1}{n!}\int_M (-nL f(t)+ R^{ch}(\omega))^2 exp((n-2)f(t))\omega^n.
$$
Its t-derivative (with $v=\frac{d}{dt}f (t)$ at $t$) of $\mathbf C(\omega(t))$ is
\begin{align*}
&\frac{d}{dt}\mathbf C(\omega(t))=\frac{1}{n!}\int_M 2(-nL f(t)
+ R^{ch}(\omega))(-L v)exp((n-2)f(t))\\ &+(n-2)(-nL f(t)
+ R^{ch}(\omega))^2 exp((n-2)f(t))v\omega^n\\
&=\frac{1}{n!}\int_M (-L)^*\Big(2exp((n-2)f(t))(-nL f(t)
+ R^{ch}(\omega))\Big)v\\ &+(n-2)(-nL f(t)
+ R^{ch}(\omega))^2 exp((n-2)f(t))v\omega^n\\
&=\frac{1}{n!}\int_M (-L)^*\Big(2exp((n-2)f(t))(-nL f(t)
+ R^{ch}(\omega))\Big)v+(n-2)R^{ch}(\omega(t)^2)exp(nf(t))v\omega^n.
\end{align*}
Note that for $\omega(t)\in \{\omega\}_1$, i.e.,
$$
V(t)=\frac1{n!} \int_M \omega(t)^n=\frac1{n!} \int_M exp(nf(t))\omega^n=1,
$$
we have
$$
V'(t)=\frac1{n!}\int_M exp(nf(t))\ f_t(t)\omega^n=0.
$$
Then the Euler-Lagrange equation for the extremal metrics ($t=1$ and $f=f(1)$) in the class $\{\omega\}_1$ is
\begin{equation}\label{EL}
(-L)^*\Big(2exp((n-2)f)(-nL f
+ R^{ch}(\omega))\Big)+(n-2)(R^{ch}(e^f\omega))^2 exp(nf)=\lambda exp(nf)
\end{equation}
where $\lambda\in \mathbb R$ is the Lagrange multiplier. Note, by integration on both sides,
$$
\lambda =\frac{n-2}{\int_M exp(nf)}\int_M (R^{ch}(e^f\omega))^2 exp(nf).
$$
If $f$ is a solution of the equation \eqref{EL}, then we call the metric $exp(f)\omega$ a \emph{Chern-Calabi extremal metric}. Given the regular initial data $f_0$ such that $\omega_0=e^{f_0}\omega\in \{\omega\}_1$. The related heat flow (i.e., Chern-Calabi extremal metric flow) is
\begin{equation}\label{EL0}
f_t=exp(-nf)(L)^*\Big(2exp((n-2)f)(-nL f
+ R^{ch}(\omega))\Big)-(n-2)(R^{ch}(\omega(t)))^2+\lambda(t)
\end{equation}
where $\lambda(t)$ is given by
$$
\lambda(t)=(n-2)\int_M(
R^{ch}(\omega(t)))^2/V(t).
$$
Since along the flow \eqref{EL0},
\begin{align*}
V'(t)&=\frac1{n!}\int_M exp(nf(t))\ f_t(t)\omega^n\\
&=\int_M(L)^*\Big(2exp((n-2)f)(-nL f
+ R^{ch}(\omega))\Big)-(n-2)(
R^{ch}(\omega(t)))^2 +\lambda(t)=0,
\end{align*}
we have $V(t)=1$, i.e., $\omega(t)\in \{\omega\}_1$ and
$$\lambda(t)=(n-2)\int_M(
R^{ch}(\omega(t)))^2=(n-2)\mathbf{C}(\omega(t)).$$
Note that along the flow,
$$
(L)^*\Big(2exp((n-2)f)(-nL f
+ R^{ch}(\omega))\Big)-(n-2)(R^{ch}(\omega(t)))^2 exp(nf)=(f_t-\lambda(t)) exp(nf).
$$
Then
$$
\frac{d}{dt}\mathbf C(\omega(t))=-\int (f_t-\lambda(t))f_t exp(nf)=-\int_M|f_t|^2exp(nf)\leq 0.
$$
Hence, $\mathbf C(\omega(t))$ (i.e., $\lambda(t)$) is non-increasing along the flow.
The standard method can always provide the short time existence of the flow. Because of the quadratic term $(R^{ch}(\omega(t)))^2$, it is a difficult heat flow.

\subsection{Statements of main results}

In below, we assume $n=2$ and then the equation \eqref{EL} is simply
$$
(-L)^*\Big(-nL f
+ R^{ch}(\omega)\Big)=0.
$$
As it is well-known via a use of maximum principle that when $\omega$ is Gauduchon, $ker\{(-L)^*\}=\{constant\}$.
Then
in the special case when $n=2$, the extremal metric equation above (so called Chern-Calabi extremal metric problem) is
$$
-nL f
+ R^{ch}(\omega)=C.
$$
Integrating over $M$, we find that
$C=\int_M R^{ch}(\omega):=\Gamma_M(\omega)$, which is the Gauduchon degree. In this case, we always can solve the extremal metric equation above to get the solution $f_e$ and the solution is unique up to adding a constant.
Note that the scalar curvature of the extremal metric $exp(f_e)\omega$ is
$$
R^{ch}(exp(f_e)\omega)=exp(-f_e)(-nL f_e+ R^{ch}(\omega))=exp(-f_e) \Gamma_M(\omega),
$$
which may not be a constant.
This gives us the result below.
\begin{thm}\label{thm1} Given a compact complex Hermitian surface $(M,\omega)$.
 Assume that $\omega$ is the Gauduchon metric.
Then there is a unique function $f_e$ satisfying
$$
R^{ch}(\omega)-\overline{R^{ch}(\omega)}=2Lf_e, \ \ \int_M e^{2f_e}\frac{\omega^2}{2}=1,
$$
and the metric $exp(f_e)\omega$ is the unique extremal to Chern-Calabi extremal metric problem  in the conformal class $\{\omega\}_1$.
\end{thm}

In below, we consider the Chern-Calabi flow on complex surfaces.
In complex dimension $n=2$, for $\omega(t)=exp(f(t))\omega$, the Calabi functional is
\begin{equation}\label{calabi0}
\mathbf C(\omega(t))=\frac12 \int_M (-nL f(t)+ R^{ch}(\omega))^2\omega^2=V |\overline{R^{ch}(\omega)}|^2+4\|L(f-f_e)\|^2
\end{equation}
and its t-derivative (with $v=\frac{d}{dt}f (t)$) is
$$
\frac{d}{dt}\mathbf C(\omega(t))=-\int_M (L )^{*}(-nL f(t)+ R^{ch}(\omega))v\omega^2.
$$
In this case, the Chern-Calabi flow on compact complex surfaces is simply given by
 \begin{equation}\label{chern-calabi0}
 f_t=exp(-2f)\Big((L)^{*}(-nL f+ R^{ch}(\omega)) \Big)
 \end{equation}
 with the initial data $\omega_0\in \{\omega\}_1$. Here we recall that $-Lu=\Delta_du+(du,\theta)_\omega$
 and
 $$
-L^{*}u=\Delta_du-(du,\theta)_\omega.
$$

Assume now $(M,\omega)$ is a Gauduchon complex surface.
we have
\begin{equation}\label{calabi}
\frac{d}{dt}\mathbf C(\omega(t))=-\int_M (L)^{*}(-nL f(t)+ R^{ch}(\omega)) v \omega^2.
\end{equation}
Since  the Gauduchon condition, $\partial \bar{\partial} \omega =0$, we can conclude that
$$
\int_{M} L(h) \omega^2=2 \int_{M} \sqrt{-1} \partial \bar{\partial} h \wedge \omega =2 \int_{M} h\sqrt{-1} \partial \bar{\partial} \omega=0.
$$

  On the Gauduchon manifold,  we have 
 the Chern-Calabi flow that
 \begin{equation}\label{chern-calabi}
 f_t=exp(-2f)\Big((L)^{*}(-nL f+ R^{ch}(\omega)) \Big)
 \end{equation}
 with the initial data $f_0$.
 Clearly by standard parabolic theory, we have the short time existence of the flow.
 We may re-write \eqref{chern-calabi} as
 $$
  \omega(t)_t= exp(-2f)(L)^{*}\Big(exp(f)R^{ch}(\omega(t)) \Big)\omega(t)
 $$
 or
 $$
 f_t=exp(-f)\Big((\Delta^{ch}_{\omega (t)})^{*}(exp(f)R^{ch}(\omega(t))) \Big).
 $$

Then we have the result about the global flow.

\begin{thm}\label{thm2} Given a compact complex Hermitian surface $(M,\omega)$.
 Assume that $\omega$ is a Gauduchon metric.
Then there is a global flow to \eqref{chern-calabi} and for any $\tau>0$ there exists a uniform constant $C=C(\tau)>0$ such that
$\|f(t)\|_{L^\infty(M)}\leq C$ for all $t>\tau$.
\end{thm}

As for the limit of the global solutions, we have the result below.
\begin{thm}\label{thm3} Given a compact complex Hermitian surface $(M,\omega)$ with the Gauduchon metric $\omega$.
Assume that there is some $C>0$ such that the global solution satisfies $\|f(t)\|_{L^\infty(M)}\leq C$ for all $t>0$.
Then the global flow converges to the solution $f_e$ of Chern-Calabi extremal problem.
\end{thm}

The plan of the paper is as follows.
 In section \ref{sect2}, we derive uniform estimate of norms of the solution $f$ in any finite time. We give a proof of global Chern-Calabi flow in section \ref{sect3}. We study the global convergence of the Chern-Calabi flow in section \ref{sect4}. In the last section, We shall discuss this higher dimensional flow. We derive the Bochner type formula and point out the possibility of the method proposed by Chang in \cite{Chang}.

\section{Estimations of the average of $f$ and norms of $f$}\label{sect2}

 Since the volume is invariant along the flow \eqref{chern-calabi}, we have that
 for any $t>0$,
$$
\frac12 \int_M exp(2f)\omega^2=1.
$$
Using Jensen's inequality and  $\int_M exp(2f)=2$, we have
$$
exp (2\int_M f)\leq 2
$$
and then $\int_M f\leq C$. We need to find the lower bound for $\int_M f$.

Note that along the flow, we have
\begin{equation}\label{star}
\frac{d}{dt}\mathbf C(\omega(t))=-\int_M|f_t|^2exp(2f) \omega^2=-\int_M |(L )^{*}(- nL f+ R^{ch}(\omega))|^2\omega^2\leq 0.
\end{equation}
Integrating \eqref{star} on $[0,T]$, where $T$ is the maximal existence time of the flow, we get
$$
\mathbf C(\omega(T)) +\int_0^Tdt\int_M |f_t|^2exp(2f)\omega^2=C(\omega(0))
$$
or
\begin{equation}\label{star1}
C(\omega(T))+\int_0^Tdt\int_M |(L )^{*}(- nL f+ R^{ch}(\omega))|^2\omega^2=C(\omega(0)).
\end{equation}

Set $H_{0}^{1}(M)=\Big\{u\in H^{1}(M)|\; \int_{M}u =0\Big\}.$ Since $(M,\omega)$ is a Gauduchon metric, $Ker \Delta^{ch}|_{H_{0}^{1}(M)}=Ker( \Delta_{d}+(d\cdot, \theta)|_{H_{0}^{1}(M)}=0$. Applying the Fredholm alternative theorem, we can conclude that, for any $g\in L^{2}(M),$ there exists an unique weak solution $u\in H_{0}^{1}(M)$ satisfying
$$
\Delta_{d} u+(du,\theta)_{\omega}=g.
$$
In particular,
\begin{equation}\label{est.u}
\|u\|_{H^{1}_{0}(M)}\leq C(\|g\|_{L^{2}}).
\end{equation}
Since $-L(u-\bar u)=\Delta_d(u-\bar u)+(d(u-\bar u),\theta)_\omega$,we have
$$
\|\Delta_{d} u\|_{L^{2}}\leq \|L(u)\|_{L^{2}}+C\|du\|_{L^{2}}\leq C(\|L(u)\|_{L^{2}}+1),
$$
where we  used $u-\bar u$ to \eqref{est.u} in the last inequality.
Recall that on the compact complex surface, we have
$$
\phi(x)-\overline{\phi}=\int_M (-\Delta_d \phi)(y)G(y,x) d\mu_\omega
$$
where $G(x,y)=G(y,x)$ is the Green function on $(M,\omega)$.
Then it holds the Adams inequality (see Theorem 1 and Theorem 2 in \cite{A} and also see \cite{B}) that there are some uniform constants $\alpha_2$ and $C=C_M>0$ such that for any $\phi\in C^\infty$ with
$L\phi\in L^2(M)$, it holds
\begin{equation}\label{adams}
\int_M exp \Big(\alpha_2\frac{(\phi-\overline{\phi})^2}{\int_M (-L \phi )^2}\Big)\leq C
\end{equation}
where
$$
\overline{\phi}=\frac{1}{Vol(M,\omega)}\int_M \phi .
$$
Note that for $\phi=f$, the uniform estimate of $\int_M (-L \phi )^2$ can be derived from the conformal invariant quantity $\int_M |R^{ch}(\omega)|^2 \omega^2$.   Then
$$
\int_M exp(2(f-\overline{f}))\leq Cexp(\int_M (-L f )^2)\leq C.
$$
By this we have
$$
2=\int_M exp (2f)=\int_M exp (2f-2\overline{f}) \cdot exp(2\overline{f})\leq C exp(2\overline{f}).
$$
This gives us that $\overline{f}\geq C$ and the constant $C$ is an uniform constant.

Set $M_{+}=\{x\in M; f(x)\geq  0\}$ and $M_{-}=\{x\in M; f(x)<  0\}$.
Since we have for any $p\in \mathbf R$,
$$
\int_M exp(p(f-\overline{f}))\leq C_pexp(\int_M (-L f )^2)\leq C_p
$$
and
$$
\int_M exp(pf)\leq C_pexp(\int_M (-L f )^2)\leq C_p.
$$
Then for $p>0$,
$$
\int_{M_+} exp(pf)\leq \int_M exp(pf)\leq C_p.
$$
and for any $p<0$,
$$
\int_{M_-} exp(pf)\leq \int_M exp(pf)\leq C_p.
$$
Hence, we have for any $p>0$,
$$
\int_M exp(p|f|)= \int_{M_+} exp(pf)+\int_{M_-} exp(-pf)\leq 2C_p,
$$
which implies that for any $p>1$, $\|f\|_{L^p(M)}\leq C_p$.
The energy identity \eqref{star1} and the estimates above implies that for $t>0$ and any $p>1$,
\begin{equation}\label{star2}
f\in L_t^2 H^4_x(M_T)\cap L^p_tH^2_x(M_T), \ \ M_T=M\times [0,T].
\end{equation}
If $T=\infty$, we are done, otherwise, $T<\infty$. We will show below that $T<\infty$ is impossible.

We now make a remark.
We can simplify the estimate above by using the function $f_e$.
We define
 $$
 v(t)=exp(f)f_t=exp(-f)\Big((L)^{*}(-2L f+ R^{ch}(\omega)) \Big).
 $$
 Define $h=f-f_e$,
 $$
 Q(\omega)=-L^{*}R^{ch}(\omega) \ \ \text{and} \ \ P=L^{*}L.
 $$
 Then we have
 $$
 v(t)=- 2exp(-f) Pf- exp(-f)Q(\omega)=-2 exp(-f) P(h).
 $$
 Hence, we have
 $~~ f_t=exp(-f)v,
 $
 and  $$
 h_t=f_t=-2exp(-2f)P(h).
 $$
 Set
$$
J(t):=\int_M |L(h)|^2d\mu.
$$
and we compute the growth of the integral $J(t)$ along the flow \eqref{chern-calabi}.

Note that
\begin{align*}
 \frac{d}{dt} J(t)& =2\int_M f_t P(h) d\mu \  \\
&= -4\int_M exp(-2f) |P(h)|^2d\mu =-\|v(t)\|_2^2.
\end{align*}
In the term of $f$, we can write it as
$
\frac{d}{dt} \mathbf C(f(t))=-4\int_M e^{2f} f_t^2d\mu.
$
Then we have
\begin{equation}\label{ell2}
J(t)+\int_0^t\|v(s)\|_2^2ds=J(0).
\end{equation}
Then
\begin{equation}\label{ell3}
\|Lf\|_2^2\leq \frac12 (\mathbf C(f_0)+\|R^{ch}(\omega)\|^2).
\end{equation}

\section{Proof of Theorem \ref{thm2}}\label{sect3}
We prove the global existence of the flow in this section.

The general idea to prove the existence of global flow is to use Moser iteration to derive the $L^\infty$ bounds of the solution and extend the flow beyond the finite maximal time to get contradiction.
We may outline the proof of this kind.
Set $u=exp(f)$. Then $\partial_iu=u\partial_if$,
$$
Lu=uL f+u|df|^2_\omega.
$$
Using \eqref{chern-sc}, we have
$$
exp(f)R^{ch}(\omega(t))=- nL f+ R^{ch}(\omega).
$$
Then
$$
-Lu=u (R^{ch}(\omega(t))u -R^{ch}(\omega))-u|df|^2_\omega\leq u (R^{ch}(\omega(t))u -R^{ch}(\omega))\leq Au,
$$
where
$$
A=|R^{ch}(\omega(t))u -R^{ch}(\omega)|\in L^2(M).
$$
Then we can use Moser iteration (see Proposition Proposition 3.4 in \cite{Chang}) to conclude that $u\in L^\infty$ up to the time $T<\infty$.
This give us the global existence of the flow.

In below, using the argument similar to \cite{B}, we mainly derive the $H^4$ norm bound for the solution. Similar bounds for higher order may also be derived.
We now prove Theorem \ref{thm2}.
\begin{proof}
 Along the flow, we set
 $$ I(t)=\int_M |P(h)|^2d\mu.
 $$
 Recall $\  P=(\Delta^{ch}_{\omega})^{*}\Delta^{ch}_{\omega}$.  Define
 $$
 Q(\omega)=-L^{*}R^{ch}(\omega).
 $$
 Note that
 \begin{align*}
 \frac{d}{dt} I(t)& =2\int_m P(h)P(f_t)d\mu \\
 &=-\int e^{f}v P(e^{-f}v)d\mu=-\int L(e^{f}v) L(e^{-f}v)d\mu.
 \end{align*}

For the estimate of the term $I$, we compute that
\begin{align*}
&\Delta^{ch}(exp(f)v) \Delta^{ch}(exp(-f)v)\\
&=\Big(\Delta_d(exp(f)v)+\exp(f)(dfv+dv,\theta) \Big)\Big(\Delta_d(exp(-f)v)+\exp(-f)(-dfv+dv,\theta)\Big) \\
&=\Big(\Delta_d v+(|df|_{\omega}^{2}+\Delta_{d}f)v+2\langle df, dv\rangle_\omega+(dfv+dv,\theta) \Big)\Big(\Delta_d v+(|df|_{\omega}^{2}-\Delta_{d}f)v\\
&\quad-2\langle df, dv\rangle_\omega+(-dfv+dv,\theta) \Big)\\
&\geq \frac12 (\Delta_d v)^2-C(|df|_\omega^2v)^2-C\langle df, dv\rangle_\omega^2-C(v\Delta_d f)^2-C|df|_\omega^2v^2-C |dv|_\omega^2.
\end{align*}

Then we have
\begin{align*}
&\int_M\Delta^{ch}(exp(f)v) \Delta^{ch}(exp(-f)v)\\
&\geq \frac12 \int_M(\Delta_d v)^2-C\int_M(|df|_\omega^2v)^2+\langle df, dv\rangle^2+(v\Delta_d f)^2+|df|_\omega^2v^2+ |dv|_\omega^2.
\end{align*}
By using the Sobolev inequality we have
$$
\|df\|_{L^8}\leq C\|f\|_{W^{2,8/3}}\leq  C \|f\|_{W^{2,2}}^{3/4} \|f\|_{C^{2}}^{1/4}\leq  C \|f\|_{W^{2,2}}^{3/4} \|f\|_{W^{4,2}}^{1/4},
$$
$$
\|v\|_{L^4} \leq C \|v\|_{L^2}^{1/2} \|v\|_{W^{2,2}}^{1/2}\leq \|v\|_{W^{2,2}},
$$
and then by Holder inequality
$$
\int_M(|df|_\omega^2v)^2\leq \|df\|_{L^8}^{4}\|v\|^{2}_{L^4} \leq C \|f\|_{W^{2,2}}^{3} \|f\|_{W^{4,2}}\|v\|^2_4.
$$
Then we have
$$
\int_M(|df|_\omega^2v)^2\leq C \|f\|_{W^{2,2}}^{3} \|f\|_{H^{4}}\|v\|_{H^{2}}\|v\|_{L^{2}}.
$$
Note that
$$
\|v\|_{W^{1,8/3}} \leq C \|v\|_{H^1}^{1/2} \|v\|_{W^{2,2}}^{1/2}\leq C \|v\|_{L^2}^{1/4} \|v\|_{W^{2,2}}^{3/4}.
$$
Then
\begin{align*}
\int_M\langle df, dv\rangle^2& \leq\|df\|^{2}_{L^8}\|dv\|^{2}_{L^{8/3}}\\
&\leq \|f\|_{H^2}^{3/2}\|f\|_{H^4}^{1/2} \|v\|_{L^2}^{1/2}\|v\|_{H^{2}}^{3/2}.
\end{align*}
The other term $\int_M (v\Delta_d f)^2$ can be estimated similarly.
\begin{align*}
\int_M (v\Delta_d f)^2& \leq \|\Delta_d f\|_{L^{4}}^2\|v\|_{L^4}^2\\
&\leq \|f\|_{H^2}\|f\|_{H^4} \|v\|_{L^2}\|v\|_{H^{2}}.
\end{align*}
Similar estimates also holds for $|df|_\omega^2v^2+ |dv|_\omega^2$.
Then we have
\begin{align*}
&\int_M(|df|_\omega^2v)^2+\langle df, dv\rangle^2+(v\Delta_d f)^2+|df|_\omega^2v^2+ |dv|_\omega^2\\
\leq &C(\|v\|_{H^2}\|v\|_2\|f\|_{H^4}+\|v\|_{H^2}^{3/2}\|v\|_2^{1/2}\|f\|_{H^4}^{1/2}+\|v\|_{H^{1}}^{2}).
\end{align*}
Then we have
\begin{align}\label{key2}
 \frac{d}{dt} I(t)& =-\frac{2}{n}\int_M exp(f)v Pf_t \  \nonumber\\
&\leq -\frac{1}{n}\int_M(\Delta_d v)^2 \nonumber\\
& +C(\|v\|_{H^2}\|v\|_2\|f\|_{H^4}+\|v\|_{H^2}^{3/2}\|v\|_2^{1/2}\|f\|_{H^4}^{1/2}+\|v\|_{H^{1}}^{2}).
\end{align}
Recall by \eqref{ell2} we have that $\int_0^T\|v\|_2^2\leq C$.
From this and H\"older inequality we then have, for small $\epsilon>0$,
$$
\|v\|_{H^2}\|v\|_2\|f\|_{H^4}+\|v\|_{H^2}^{3/2}\|v\|_2^{1/2}\|f\|_{H^4}^{1/2}
\leq \epsilon \|v\|_{H^2}^2+C_\epsilon \|f\|_{H^4}^2\|v\|_2^2.
$$
Hence, by the elliptic estimate $\|v\|_{H^2}^2\leq C(\|\Delta_dv\|_2^2+\|v\|_2^2)$, we have for some uniform $c>0$,
$$
\frac{d}{dt} I(t)+ c\|v\|_{H^2}^2\leq C(\|f\|_{H^4}^2+1)\|v\|_2^2.
$$
Recall $P(f)=P(h)+P(f_e)$.
By the elliptic estimate \eqref{ell} and \eqref{ell3} we have
$$
\|f(t)\|_{H^4}\leq C(1+I(t))
$$
and then
$$
\frac{d}{dt} \log (I(t)+1)\leq C\|v\|_2^2.
$$
Hence, we have
$$
(I(t)+1)\leq C, \ \ \ \forall t\in (0,T],
$$
which give the desired uniform $H^4$ norm control.

By Sobolev inequality
 for any $p>1$,$$f\in H^4(M)\hookrightarrow W^{3,4}(M)\hookrightarrow W^{2,p}(M).$$
By Sobolev inequality again, we know that $f\in C^{1,\alpha}$ (for some $\alpha\in (0,1)$) is uniformly bounded in $M_T$. In particular, we have some uniform constant $C$ such that $0<C^{-1}\leq e^{\pm f}\leq C<\infty$.
From \eqref{chern-sc}, we also know that
$R^{ch}(\omega(t))\in H^2$. By Sobolev inequality we know that $R^{ch}(\omega(t))\in L^p$ for any $p>1$.

Set $a(x,t)=e^{-f}$. Then we have the parabolic equation
$$
h_t+a(x,t)Ph=0
$$
and use the interior estimate in Theorem 2.2(I) in \cite{DZ2015} with $d=4$, we can conclude that for any metric ball $B_r$ in coordinate systems and $t_0>r^4$ $Q_r=(t_0-r^4,t_0)\times B_r$, there holds the uniform estimate
$$
\|h_t\|_{C_x^{\alpha}(Q_r)}+\|\nabla^4h\|_{C_{(x,t)}^{\alpha/4,\alpha}(Q_r)}\leq C \|h\|_2.
$$
With this understanding, we can obtain the uniform estimate of any higher derivatives of the solution $f$ for $t>\tau>0$. Then $T=\infty$, that is, the solution is a global one.
This completes the proof.

\end{proof}
We remark that we may also use the similar method as above to get higher order norm estimate. Let $H=\Delta_\omega^{ch} f$. Then $H\in H^2$. Note that
$$
L exp(-2f)= -2exp(-2f) H+4exp(-f)|\partial f|^2, \ \ H_t=L(f_t)
$$
and using the equation \eqref{chern-calabi}, we obtain
\begin{equation}\label{more}
H_t=L \Big(exp(-2f)(-n(\Delta_\omega^{ch})^{*}H+Q )\Big).
\end{equation}
Note that $(\Delta_\omega^{ch})^{*}=\Delta_\omega^{ch}+2(d\cdot,\theta)_{\omega}$ and
\begin{align*}
L[exp(-2f)u]&=g^{i\bj}\partial_i(-2exp(-2f)f_{\bj} u+exp(-2f)u_{\bj})\\
&=exp(-2f)\Big(4|\partial f|^2u-2g^{i\bj}(f_{\bj} u_i+u_{\bj} f_i) -\Delta_\omega^{ch}u\Big).
\end{align*}
The equation \eqref{more} is of the similar type as the equation \eqref{chern-calabi}. The extra term $2g^{i\bj}(f_{\bj} u_i+u_{\bj} f_i)$ with $u=(\Delta^{ch})^{*} H$ can be treated using the Bochner type formula  \eqref{star8} derived in section \ref{sect3}.
 Repeating the argument above to \eqref{more} or applying standard parabolic theory to the equation \eqref{chern-calabi}, we know that $H\in H^4$, which then implies that $f\in C^{3,\alpha}$ and $R^{ch}(\omega(t))\in C^{1,\alpha}$. Again, we may get higher order derivative estimates for $f$, which implies that we can extend the solution beyond time $T$. Hence $T=\infty$.

\section{Convergence of Chern-Calabi flow}\label{sect4}

 In this section, we consider the asymptotic behavior of the global Chern-Calabi flow obtained in Theorem \ref{thm2}.

 We now prove Theorem \ref{thm3}.
 \begin{proof}
 Since we have a global solution to the Chern-Calabi flow \eqref{chern-calabi0}, by the energy identity \eqref{star1}
 we have $t_i\to\infty$ such that
 $$
 \int_M |f_t|^2exp(2f(t_i))\omega^2\to 0.
 $$
  Assume that there is some $C>0$ such that $\|f(t)\|_{L^\infty(M)}\leq C$ for all $t>0$. Then all constant $C(T)$ in \eqref{key2} do not depend on $T$ and they unform constants. Then we can obtain from \eqref{key2} that
  $$
\frac{d}{dt} I(t)\leq C(\|f\|_{H^4}^2+1)\|v\|_2^2\leq C(I(t)+1)\|v\|_2^2.
$$
Choose $T>0$ large such that $C\int_T^\infty \|v\|_2^2<1/2$. Then we have for any $t>T$,
 $ \ \
\frac{d}{dt} \log (I(t)+1)\leq C(\|v\|_2^2) \ \
$
 and then
 $ \
 \log \frac{(I(t)+1)}{(I(T)+1)}\leq 1/2, \
 $
 which implies the uniform $C^{1,\alpha} (M)$ of $f(t)$ for $t\geq T$. By bootstrap argument, we know that $f(t)$ is uniformly bounded in
 $C^{4,\alpha}(M)$. Hence we have the limit $f_\infty(x)$ of $f(t_i)$ in  $C^{4,\alpha}(M)$ and $f_\infty$ satisfies the limiting equation
 $$
 (L )^{*}(-nL f_\infty(x)+R^{ch}(\omega))=0
 $$
 Then we have
 $\ \
 -nL f_\infty(x)+ R^{ch}(\omega)=\overline{R^{ch}(\omega)}.
 $
 Since the limit $f_\infty(x)$ is the unique function $f_e$, we can conclude the flow converges to $f_e$ as $t\to \infty$.
 This completes the proof of Theorem \ref{thm3}. \end{proof}

\section{Other problems and Bochner formula for Chern connection}\label{sect5}

In this section, we consider the relevant Bochner formula for the elliptic method of the higher dimensional geometric flow \eqref{EL0} on the Gauduchon compact complex manifolds. For higher dimensional flow \eqref{EL0}, there is finite time blow-up solution and the result may be presented in latter chances.
By assuming the flow $f\in L^\infty$ and $f\in H^4$ for the higher dimensional Chern-Calabi extremal metric heat flow \eqref{EL0}, we can obtain the global flow.
We also propose a fourth order elliptic problem, which is related to the Kazdan-Warner problem in two dimensional Riemannian geometry.

As we have said before, it is possible to use similar argument as in \cite{Chang} to derive the uniform bound of the solution. The only part missing the Chern connection in the local Hermitian geometry is the Bochnor type formula.

Recall that
in local coordinates $z = (z_1, \ldots, z_n)$, for $u\in C^\infty$,
$z_j = x_j + \sqrt{-1} y_j$, we use notations such as $\partial_k=\frac{\partial}{\partial z_k}$ and
\[ u_i = \nabla_{\frac{\partial}{\partial z_i}} u, \;
   u_{ij} = \nabla_{\frac{\partial}{\partial z_j}}
             \nabla_{\frac{\partial}{\partial z_i}} u, \;
   \mbox{etc} \]
and
$$
\langle X,Y\rangle=g(X,\overline{Y}), \ \ |\partial u|^2=g^{i\bj}u_iu_{\bj}, \ \ |\partial \overline{\partial} u|^2=g^{k\bl}g^{i\bj}u_{i\bl}u_{k\bj}, \ \ |D^2 u|^2=g^{k\bl}g^{i\bj}u_{ik}u_{\bj\bl}.
$$
 For second derivatives of $u$, we have
 $$
 u_{i\bj} - u_{\bj i}  = 0, \;\;
 u_{ij} =\partial_j \partial_iu-\Gamma_{ji}^k \partial_k u, \ \ \  u_{ij} - u_{ji}  = T_{ij}^l u_l,
 $$
 where $T=(T_{ij}^l)$ is the torsion of the Chern connection.

 Recall here that the curvature of Chern connection is given by
$$
 R (u, v) w  = \nabla_u \nabla_v w - \nabla_v \nabla_u w - \nabla_{[u,v]} w.
$$
In local coordinates $(z_i)$, we have
\begin{equation}
\label{chern-curv}
\begin{aligned}
 R_{i\bj k\bl} \equiv &\, R \Big(\frac{\partial}{\partial z_i},
 \frac{\partial}{\partial \bz_j}, \frac{\partial}{\partial z_k},
 \frac{\partial}{\partial \bz_l}\Big)=g(R (\frac{\partial}{\partial z_i},
 \frac{\partial}{\partial \bz_j}) \frac{\partial}{\partial z_k}, \frac{\partial}{\partial \bz_l}) \\
 = \,& - g_{m \bl} \frac{\partial \Gamma_{ik}^m}{\partial \bz_j}
 = - \frac{\partial^2 g_{k\bl}}{\partial z_i \partial \bz_j}
       + g^{p\bq} \frac{\partial g_{k\bq}}{\partial z_i}
                  \frac{\partial g_{p\bl}}{\partial \bz_j}.
         \end{aligned}
\end{equation}
Then we have
$$
u_{i k \bj}=(u_{ki} )_{\bj} + (T_{ik}^l u_l)_{\bj}=u_{ki\bj}+(T_{ik}^l)_{\bj}u_l+T_{ik}^l u_{l \bj}.
$$
Since
$~~
u_{i \bj k} - u_{i k \bj} = \,  - g^{l\bm} R_{k \bj i \bm} u_l,
$
we then have
$$
u_{i \bj k}=u_{ki\bj}+(T_{ik}^l)_{\bj}u_{l}+T_{ik}^l u_{l \bj}-g^{l\bm} R_{k \bj i \bm} u_l
$$
Hence
$$
(\Delta_{\omega} u)_k=(g^{i \bj})_ku_{i\bj }+g^{i \bj}(u_{ki\bj}+ (T_{ik}^l)_{\bj}u_{l}+T_{ik}^l u_{l \bj}-g^{l\bm} R_{k \bj i \bm} u_l).
$$
which is,
\begin{equation}\label{star8}
(\Delta_{\omega} u)_k=\Delta_{\omega}^{ch}u_{k}+(g^{i \bj})_ku_{i\bj }+  g^{i \bj}(T_{ik}^l)_{\bj}u_{l}+ g^{i \bj}T_{ik}^l u_{l \bj}-g^{l\bm} g^{i \bj}R_{k \bj i \bm} u_l.
\end{equation}

To derive the Bochner formula, we need to more relations for the smooth function $u\in C^\infty(M)$:
\begin{equation}
\label{ma6}
\left\{
\begin{aligned}
u_{i \bj \bk} - u_{i \bk \bj} = \, & \ol{T_{jk}^l} u_{i\bl},  \ \ \  u_{i \bj k} - u_{i k \bj} = \,  - g^{l\bm} R_{k \bj i \bm} u_l, \\
 u_{i j k} - u_{i k j} = \, & g^{l\bm} R_{jki\bm} u_l + T_{jk}^l u_{il}.
\end{aligned}
 \right.
\end{equation}

Using the relation
$$
R_{i \bj k\bl} - R_{k \bj i\bl}
   =  g_{m \bl} \frac{\partial T_{ki}^m}{\partial \bz_j}
   = g_{m \bl} \nabla_{\bj} T_{ki}^m
   $$
   and \eqref{star8},
we can obtain
\begin{equation}\label{ma7}
(\Delta_{\omega} u)_k=\Delta_{\omega}^{ch}u_{k}+(g^{i \bj})_ku_{i\bj }+ g^{i \bj}T_{ik}^l u_{l \bj}-g^{l\bm} g^{i \bj}R_{i \bj k \bm} u_l.
\end{equation}
Then
\begin{align*}
\Delta^{ch}|\partial u|^2&= g^{k \bl} \nabla_k\nabla_{\bl} |\partial u|^2\\
&=  g^{k \bl} \Big( \nabla_k g^{i\bj}[\nabla_{\bl}u_iu_{\bj}+u_i\nabla_{\bl}u_{\bj} ] \Big) \\
&=  g^{k \bl} \Big(  g^{i\bj}[u_{i \bl k}u_{\bj}+u_{i \bl }u_{\bj k}+u_{ik}u_{\bj \bl}+\nabla_ku_{\bj \bl} u_{i} ] \Big) \\
&=  g^{k \bl} g^{i\bj}u_{i \bl k}u_{\bj}+g^{k \bl} g^{i\bj}u_{\bj   \bl k}u_{i}+|\partial \overline{\partial} u|^2+|D^2u|^2.
\end{align*}
With this understanding, by using \eqref{ma6}, \eqref{ma7} and by using Cauchy-Schwarz inequality to treat the torsion parts, we then obtain the Bochner type inequality below.
For any $\epsilon>0$, there exists a constant $C_\epsilon>0$ such that for any smooth function $u$,
\begin{equation}\label{bochner}
\Delta^{ch}|\partial u|^2\geq 2Re \langle \partial u, \partial \Delta^{ch} u\rangle+(1-\epsilon)|D^2u|^2+(1-\epsilon)|\partial\overline{\partial} u|^2-C_\epsilon |\partial u|^2,
\end{equation}
which can also be written as
\begin{equation}\label{bochner1}
-\frac12\Delta_d|du|^2\geq - \frac12 \langle d u, d \Delta_d u\rangle+(1-\epsilon)|D^2u|^2+(1-\epsilon)|\partial\overline{\partial} u|^2-\frac12 C_\epsilon |du|^2,
\end{equation}
Here we used
$$
\langle d|\partial u|^2,\theta\rangle\leq  C_{\epsilon} |du|^2+\epsilon|D^2u|^2+\epsilon|\partial\overline{\partial} u|^2
$$
Similarly, since there is a torsion term, we may also have Bochner type inequality as below.
\begin{equation}\label{bochner2}
-\frac12\Delta_d|du|^2\leq - \frac12 \langle d u, d \Delta_d u\rangle+(1+\epsilon)|D^2u|^2+(1+\epsilon)|\partial\overline{\partial} u|^2+\frac12 C^{\prime}_\epsilon |du|^2.
\end{equation}

We now consider the higher order derivative regularity of the flow \eqref{EL0}.
Recall the assumption that $f\in L^\infty$ and $f\in H^4$. Then
$f\in W^{3,4}\hookrightarrow W^{2,p}$ for any $p>1$,
$R^{ch}(\omega(t))\in H^2$, and $R^{ch}(\omega(t))\in W^{1,\frac{2n}{n-2}}$.
 By Sobolev embedding theorem, we know that for any $p>4$, we have
$f\in W^{2,p}\hookrightarrow C^{1,\alpha}$. Then we know that the nonlinear term in \eqref{EL0}
$$
-(n-2)(-L f
+ R^{ch}(\omega))^2 exp((n-2)f)\in H^2\in L^p
$$
for any $p>1$,
and $f$ is smooth via a bootstrap argument, which implies that we can extend the solution beyond $T$.

We want to propose a question similar to Kazdan-Warner problems on two dimensional Riemannian manifolds. It is interesting to study the prescribed curvature like problem in the following formulation. Given some nontrivial smooth function $K(x)$ on $M$. Is there a solution to the "curvature" equation
 \begin{equation}\label{soliton1}
 (L)^* (-L\lambda(x)+ R^{ch}(\omega))=K(x) e^{2\lambda(x)}, \ \ \text{on}\ \ M.
 \end{equation}
 Recall we have set $Q(x)=(L)^{*} R^{ch}(\omega)(x)$ and note that $\int_M Q(x)=0$. Then the necessary condition to have a solution $u$ is
 $ \ \
 \int_M K(x) e^{2\lambda(x)}=0.\ \
 $
 Hence, we ask the function $K(x)$ changes signs on $M$.
 Define on $H^2$ the functional
 $$
 J(\lambda)=\frac12 \int_M\Big( |L \lambda(x)|^2+K(x) e^{2\lambda(x)}\Big)-\int_M Q(x)\lambda(x).
 $$
 Then \eqref{soliton1} is the Euler-Lagrange equation of the functional $J(\cdot)$ on $H^2$.

 To find a solution to \eqref{soliton1}, we may first solve
 $$
 -L\phi=R^{ch}(\omega)-\overline{R^{ch}(\omega)}
 $$
 and it is standard to obtain the unique solution $\phi$ with $\int_M \phi(x)=0$. Inserting $R^{ch}(\omega)(x)=-L\phi(x)+\overline{R^{ch}(\omega)}$ into \eqref{soliton1} we obtain the equation
$$
 -(L)^*L(\lambda(x)+\phi(x) )=K(x) e^{2\lambda(x)}, \ \ \text{on}\ \ M.
$$
 Set $u=\lambda(x)+\phi(x)$. Then we transform the equation \eqref{soliton1} into
  \begin{equation}\label{soliton2}
 -(L)^*L u(x) =\widetilde{K(x)} e^{2u(x)}, \ \ \text{on}\ \ M,
  \end{equation}
 where
 $
 \widetilde{K}(x)=K(x) e^{-2\phi(x)}.
 $
 Still the necessary condition to solve the equation \eqref{soliton2} is the solution $u$ satisfies
 $ \ \
 \int_M\widetilde{K}(x) e^{2u(x)}=0. \ \
 $
 One can not guess that it can be solved as in Kazdan-Warner problem on 2-dimensional Riemannian manifolds.
  So we leave this problem open for interested readers.

\vspace{0.3cm}

{\small{\bf{Acknowledgments:}} {Part of this work was done during Ma's visit to Tianyuan mathematical center in southwest China in January of 2025. The authors would like to thank Professors Jixiang Fu, Xingwang Xu, and Bin Zhou for their helpful discussions. This work was reported at the Workshop on Geometric Analysis held in Wuhan from March 29 to 30, 2025. In May 2025, we learned  that Professors Xiaokui Yang and Kaijie Zhang \cite{YZ} had obtained similar result to our Theorem 1 of this paper and we would like to thank  Mr. H.Y. Chen for telling us the reference \cite{YZ}.

}

\section*{Declarations}
\subsection*{Funding}

\subsection*{Conflicts of interests}
The authors have no relevant financial or non-financial interests to disclose.
\subsection*{Availability of data and material}
Not applicable.
\subsection*{Code availability}
Not applicable.

\end{document}